\documentclass[12pt]{article}  
\usepackage{amsmath,amssymb,latexsym}

\newtheorem{thm}{Theorem}
\newtheorem{lemma}[thm]{Lemma}
\newtheorem{corollary}[thm]{Corollary}

\def\A{{\mathcal A}}
\def\F{{\mathcal F}}
\def\R{{\mathcal R}}

\def\GAP{{\sf GAP}}
\def\Magma{{\sc Magma}}

\newenvironment{proof}{{\sc Proof}:}{\hfill $\Box$} 

\begin{document}

\title{On complexity of multiplication \\ in finite soluble groups}
\author{M.F.\ Newman and Alice C.\ Niemeyer}
\date{27 August 2014}

\maketitle  

\begin{abstract}
We determine a reasonable upper bound for the complexity of collection from 
the left to multiply two elements of a finite soluble group 
by restricting attention to certain polycyclic presentations of the group.
As a corollary we give an upper bound for the complexity of collection from the left 
in finite $p$-groups in terms of the group order.
\end{abstract}


\smallskip
\noindent
2000 MSC Primary: 68Q25; 20D10, 68W30

\smallskip
\centerline{\em Dedicated to the  memory of  \'Akos Seress}

\section{Introduction}
In studying groups using computers it is important to have
practical programs for multiplication of elements.
For finite soluble groups given by (finite) polycyclic presentations 
this involves having practical programs for
collection relative to a polycyclic presentation. 
See \S2 for a description.
Since the  work of Vaughan-Lee \cite{VL} and 
that of Leedham-Green and Soicher \cite{LGS} 
it has been known that collection from the left works well in practice.
Collection from the left is the basis for multiplication in 
the computer algebra systems 
{\GAP} \cite{GAP} and {\Magma} \cite{MAGMA}.

The cost of various collection strategies has been discussed in several papers, 
see for example \cite{LGS}, \cite{Gebhardt} and \cite{H2}. 

Leedham-Green and Soicher compared the performance of some
collection strategies and 
did some complexity analysis on collection from the left for finite $p$-groups.
We address a question they raised (\cite{LGS}, p.675)
of finding for finite $p$-groups and, more generally finite soluble groups, 
polycyclic presentations which from a complexity point of view
are favourable for collection from the left. 

Gebhardt \cite{Gebhardt} extended the investigation of collection from the
left to arbitrary polycyclic presentations. In particular, he
substantially improved performance by modifying collection from
the left to deal more effectively with large powers.
His programs are the basis for the multiplication available in {\Magma}.

H\"ofling \cite{H2}  considered various favourable presentations.  

In this paper we introduce a new kind of favourable presentation.
Using these allows us to give an accessible complexity analysis 
for collection from the left. 
The theorem is stated in \S2 and proved in \S3.
Our favourable presentations come from polycyclic series which refine 
series of normal subgroups with abelian sections such as the derived series.
In practice, a better way of handling powers is by using repeated squaring 
as described by Gebhardt (\cite{Gebhardt}, Section 4).

Cannon et al. \cite{CEL} consider other special
presentations in relation to  questions about finite soluble groups.

\section{Preliminaries and favourable polycyclic presentations}

We begin by recalling some terminology and notation.

A \emph{finite polycyclic presentation} is a presentation
\mbox{$\{\A \mid \R \}$} where 
\linebreak
$\A = \{ a_1,\ldots, a_m  \}$ 
and $\R$ consists of  relations in $\A$ of the form

\begin{tabular}{ll}
$a_i^{e_i}= v_{ii}$ &for $1 \le i \le m$\\
$a_ja_i= a_iv_{ij}$ & for $ 1\le i<j\le m$,\\
\end{tabular}  \\
where $e_i$ is a positive integer for $1 \le i \le m$  and
$v_{ij}$ is a word in $\{a_{i+1}, \ldots, a_m\}$ for $1 \le i \le j \le m.$

In this context it suffices to work only with non-negative words in $\A,$ 
that is, words involving only letters from $\A$ but not their inverses. 
The order of the generators matters; we take $a_1 < \dots < a_m.$ 
As usual we use the abbreviation $a^\alpha$ for the concatenation of 
$\alpha$ copies of $a$. The words $a_1^{\alpha_1} \cdots a_m^{\alpha_m} $ 
for integers $\alpha_i$ with $0 \le \alpha_i < e_i$ for $1 \le i \le m$
are the {\em normal}  words in $\A.$
We take the right-hand sides of the relations in $\R$ to be normal words.
The left-hand  sides of the relations in $\R$  are precisely  the 
\emph{minimal non-normal} words  in $\A$.   

Every finite polycyclic presentation $\{\A\mid \R\}$ defines a
finite soluble group with order dividing $e_1\cdots e_m.$ 
It is well-known that every finite soluble group has a finite polycyclic
presentation.  A polycyclic presentation for a finite soluble group $G$ is
\emph{consistent} if $|G| = e_1\cdots e_m$.  In this case,
every element of $G$ can be written uniquely as a normal word.

Given a non-normal word $w$  in $\A$ a \emph{collection step} replaces
a minimal  non-normal subword of $w$  with the right-hand  side of the
corresponding  relation in  $\R$.  Collecting  a word  in $\A$  is the
application of a sequence of collection steps starting at the word.  Collecting
a word  in $\A$ {\em  from the left}  is the collection in  which each
step  replaces the left-most  minimal non-normal  subword of  the word
being collected in that step. A more detailed discussion of collection
and, more generally, rewriting can be found in Sims  \cite{Sims}.

We measure complexity in a more conventional way than
Leedham-Green and Soicher \cite{LGS}.
We estimate the number of collection steps required to collect the
concatenation of two normal words to a normal word with respect to
particular `favourable' polycyclic presentations, and 
estimate the length of words which occur during such a collection. 

A finite polycyclic presentation ${\F}= \{\A\mid \R\}$ 
will be called \emph{favourable} 
if there is a positive integer $d$ and a non-decreasing, surjective function 
$\delta : \A \rightarrow \{1, \ldots, d\}$ such that 
either $v_{ii}$ is $a_{i+1} $ or
$v_{ii}$ is a normal word in $\{a_k,...,a_m\}$ with $\delta(a_k) > \delta(a_i)$ 
and 
for $j > i$ each $v_{ij}$ is a normal word  and 
equal either to $v^*_{ij}$ or to $a_j v^*_{ij}$ where $v^*_{ij}$ is a normal word 
in $\{a_k,...,a_m\}$ with $\delta(a_k) > \delta(a_i)$. 
The integer $d$ is the \emph{soluble bound} of $\F$.

\begin{lemma}\label{lem:fav}
Every finite soluble group has a favourable polycyclic presentation.
\end{lemma}

\begin{proof}
Let $G$ be a finite soluble group. 
Let $d$ be the derived length of $G$ and let 
$G = G^{(0)} > G^{(1)} > \cdots >  G^{(d-1)} >  G^{(d)} = \langle 1 \rangle$ 
be the derived series of $G$. 
For $1 \le \ell \le d$ 
choose a subset $\A_{\ell}$ in $G$ 
so that $\A_{\ell} G^{(\ell)}$ is a minimal generating set for $G^{(\ell-1)}/G^{(\ell)}$
and put $\A = \A_1 \cup \dots \cup A_d$.
A function $\delta$ is defined on $\A$ by $\delta(a) = \ell$ for $a \in \A_{\ell}$.  
The polycyclic presentation built on $\A$, in the usual way, is favourable.  
\end{proof}

\medskip

Write  $\A = \{a_1,\dots,a_m\}$ and, for $a_i$ in $G^{\ell}$,
write $e_i$ for the order of $a_i$ modulo $G^{\ell+1}.$
A favourable presentation obtained from a group as in the proof is consistent
and has all $e_i >1$. 
We assume henceforth that the finite polycyclic presentations that occur 
have these properties.
Note that a finite cyclic group has a favourable presentation on  one
generator  and one relation. A finite metacyclic group  may not have a
favourable presentation on two generators but then it has a 
favourable presentation on three generators.
The normal words have length at most $e_1-1+...+e_m - 1$ which we
denote $N$.

For a finite $p$-group  we can  take the $e_i$ to be equal to $p$  
by replacing the minimality condition in the definition of $\A$
with choosing $\A$ via a polycyclic series
which is a composition series refining the derived series.

\medskip

\noindent
{\bf Theorem}
\emph{
Let ${\F}= \{\A\mid \R\}$ be a favourable polycyclic presentation with
soluble bound $d$ and maximum normal word length $N$. 
Every concatenation of two normal words in ${\A}$ can be collected
from the left to a normal word in at most $N^{3d-1}$ steps.  
All words occurring in the course of this collection have
length at most $2N$ when $d=1$ and at most $2(d-1) N^2$ otherwise.
}

\medskip
It remains an open question whether every finite soluble group has
a finite polycyclic presentation 
and a collection using that presentation for which the number of
collection steps
is polynomial in the size of the input; which we are measuring by $N$.

\section{Proof of the theorem} 

There is nothing to prove if either word is trivial. 
Let $u$ and $w$ be non-trivial normal words in $\A.$ 
Let $a_s$ be the first letter of $w$ 
and write $w = a_s w_2$.
Write $u = u_1 a_s^{\alpha_s} u_2$ 
 where $u_1$ is a word in $\{a_1,...,a_{s-1}\}$ and
$u_2$ is a word in $\{a_{s+1},...,a_m\}$. 

In collecting $uw$ from the left the subword $u_1$ is never modified and 
the subword $w_2$ is not involved until 
collection of the subword $a_s^{\alpha_s} u_2 a_s$ 
to a normal word has been completed. The theorem is a consequence of
showing that the collection of  subword $a_s^{\alpha_s} u_2 a_s$ to a normal
word takes at most $N^{3d-2}$ steps and that, in context, the length
of words occurring during the collection is at most 
$2N$ when $d=1$ and at most $2(d-1) N^2$ otherwise.

For $d=1$ collecting the subword $a_s^{\alpha_s} u_2 a_s$ to a normal word 
takes at most $N$ steps and the length of words occurring during the collection 
is at most $2N$.
Since the length of $w$ is at most $N$, the number of steps 
in the collection of $uw$ to a normal word is at most $N^2$. 
The length of words occurring during the collection remains at most $2N$.

For $ d > 1$ we define the \emph{derived} presentation 
${\F}' = \{\A'\mid\R'\}$ of $\F$ as follows.
Put $\A' = \{a \in \A\mid \delta(a) > 1 \}$. 
Let $\R'$ denote the subset of all relations in
$\R$ whose left-hand sides involve only generators in $\A'$. 
Define the function $\delta' : \A' \rightarrow 
\{1, \ldots, d-1\}$ by $\delta'(a) = \delta(a)-1.$ 
Clearly, $\delta'$ is non-decreasing and surjective. 
Then ${\F}'$ is a favourable presentation with soluble bound $d-1$.
Let $N'$ denote the maximum length of a normal word in ${\F}'$. 
By induction the least upper bound $\sigma'$ on the number of steps
required to collect the concatenation of two normal words in 
${\A}'$ to a normal word is at most $N'^{(3d-4)}$ 
and the least upper bound $\lambda'$ on 
the length of the words occurring during the collection
is at most $2N'$ when $d=2$ and at most $2(d-2) N'^2$ otherwise.
Note $N' \le N-1$.

It will be convenient to write $\Pi_1$ for a normal word in $\A'$ and 
$\Pi_r$ for the concatenation of $r$ normal words in $\A'$; 
empty words are allowed. 
The following lemma is obvious.

\begin{lemma}\label{alpha}
For $ r \ge 2$, collection from the left of \ $\Pi_r$ to a  normal word takes
at most $(r-1) \sigma'$ steps and the length of words occurring 
is at most $\lambda' + (r-2)N'.$
\end{lemma}

\begin{lemma}\label{eta}
For $\delta(a) =1 $ collection from the left of $\Pi_1 a \Pi_1$ 
to a normal word takes at most $N'  + N'\sigma'$ steps and 
the length of words occurring is at most $1 + \lambda' + (N'-1)N'  $.
\end{lemma}

\begin{proof}
The collection begins with at most $N'$ steps giving $a \Pi_{N'+1}.$ 
All the words during these steps have length at most  $1+(N'+1)N'.$
Using Lemma~\ref{alpha} the collection is completed in at most  
$N' \sigma'$ further steps during which the length of words is at most 
$1 + \lambda' + (N'-1)N' $.
\end{proof}

\begin{lemma}\label{beta}
Collection from the left of the subword
$a_s^{\alpha_s}u_2^{\vphantom{\alpha_s}}a_s^{\vphantom{\alpha_s}}$ of
$u_1a_s^{\alpha_s}u_2 a_sw_2$ to a normal word takes at most 
$N^{3d-2}$ steps and the length of words occurring is at most $2(d-1)N^2$.
\end{lemma}

\begin{proof}
Write $u_2 = a_{t_1} \dots a_{t_h} \Pi_1$
where $\delta(a_{t_1}) = \dots = \delta(a_{t_h})=1$; then $h \le N-N'-1$.
It takes at most $N-1$ steps to move $a_s$ past $u_2$ giving 
$$u_1 a_s^{\alpha_s} a_s a_{t_1}\Pi_1\dots a_{t_h}\Pi_1 \Pi_{N'} w_2.$$
The length of this word is at most $2N + (N-1)N'$.
When $\alpha_s < e_s-1$ 
the rest of the collection of  $a_s^{\alpha_s} u_2 a_s$ to a normal word
consists of at most $N-N'-2$ stages using Lemma~\ref{eta}
followed by using Lemma~\ref{alpha} to collect $\Pi_{N'+1}$.
The collection takes at most 
$N-1+(N-N'-2)(N'+N'\sigma') + N'\sigma'$ steps. 
This is at most $N^{3d-2}$.
The length of the words occurring is at most
$2N  + (N-3)N' + 1 + \lambda' + (N'-1)N'$ when the lengths of $u_1,w_2$
are included; so is at most $\lambda' + 2N^2$.
When $\alpha_s = e_s-1$ the next step
replaces $a_s^{\alpha_s} a_s$ by $\Pi_1$ or $a_{s+1}.$ 
The length is  at most $2N+NN'$.
In the first case the rest of the collection might need $N-N'-1$ uses of
Lemma~\ref{eta} but still gives the stated result.
In the second case it may be necessary to use up to $h$ power relations
collecting $a_t^{e_t}$ between uses of Lemma~\ref{eta}. Eventually one of these has the form 
$a_t^{e_t} = \Pi_1.$ Then the rest of the collection using Lemma~\ref{eta} enough times followed by
using Lemma~\ref{alpha} gives the result.
The length of the words occurring is at most
$2N  + (N-2)N' + 1 + \lambda' + (N'-1)N'$ when the lengths of $u_1,w_2$ are included; 
so is at most $\lambda' + 2N^2$.
\end{proof}

\medskip

Since the length of $w$ is at most $N$, it follows from Lemma~\ref{beta} that
the number of steps in the collection of $uw$ to a normal word is at most 
$N^{3d-1}$ and the length of words occurring is at most $2N$ when $d=1$ and 
at most $2(d-1) N^2$ otherwise.
 \hfill $\Box$

\begin{corollary}
A finite $p$-group $G$ 
has a favourable presentation with respect to which
every concatenation of two normal words can be collected from the
left in at most
$((p-1)\log_p |G|)^{3 \log_2\log_p |G| +1}$ steps.
\end{corollary}

\begin{proof}
Let the order of $G$ be $p^n$.
Let ${\F}= \{\A\mid \R\}$ be a favourable polycyclic presentation for
$G$ obtained by using a composition series refining the derived series.
The maximum normal word length for $\F$ is at most $(p-1)n.$ 
Moreover, the soluble bound $d$ for $\F$ can be chosen to be the derived length of $G$. 
By a well-known result of P. Hall 
$d$ is at most $1 + \log_2 (n-1)$ (e.g. \cite{Sims}, Corollary 9.1.11). 
So by the theorem the number of steps in collection from the left is at most
$((p-1)\log_p |G|)^{3 \log_2\log_p |G| +1}.$
\end{proof}

\subsection*{Acknowledgements}

We thank a referee for a careful and helpful report.
The second author acknow\-ledges the support of  
ARC Discovery grants  DP110101153 and DP140100416.


\begin{tabbing}
444444445456789012345123454555555\=45555555555555454444444444444444444\kill
Mathematical Sciences Institute \> School of Mathematics and Statistics\\
Australian National University \> University of Western Australia \\
Canberra ACT 0200 \> Crawley WA 6009\\
Australia \> Australia \\
newman@maths.anu.edu.au \> Alice.Niemeyer@uwa.edu.au
\end{tabbing}

\end{document}